\documentclass[twoside]{article}       

\usepackage[utf8]{inputenc}
\usepackage[british]{babel}
\usepackage[T1]{fontenc}
\usepackage{amsfonts}
\usepackage{bm}
\usepackage{amsmath}

\DeclareMathOperator{\cov}{cov}
\DeclareMathOperator{\var}{\mathsf{Var}}

\usepackage{amsthm}

\theoremstyle{plain}
\newtheorem{lemma}{Lemma}
\newtheorem{theorem}{Theorem}
\newtheorem{corollary}{Corollary}

\theoremstyle{definition}
\newtheorem{definition}{Definition}

\theoremstyle{remark}
\newtheorem{remark}{Remark}

\begin{document}

\title{Time-Varying Isotropic Vector Random Fields on Compact Two-Point Homogeneous Spaces}

\author{Chunsheng Ma\thanks{Department of Mathematics, Statistics, and Physics, Wichita State University, Wichita, Kansas 67260-0033, USA,
              Tel.: +316-978-3941, Fax: +316-978-3748,
              email: \texttt{chunsheng.ma@wichita.edu}.} \and Anatoliy Malyarenko\thanks{Division of Applied Mathematics, Mälardalen University, SE-721 23 Västerås, Sweden, Tel.: +46-21-10-70-02, Fax: +46-21-10-14-00,
email: \texttt{anatoliy.malyarenko@mdh.se}.}}

\date{\today}

\maketitle

\begin{abstract}
A general form of the covariance matrix function is derived in this paper for a vector random field that is isotropic and mean square continuous on a
compact connected two-point homogeneous space and stationary on a temporal domain. A series representation is presented for such a vector random field, which involve Jacobi polynomials and the distance defined on the  compact two-point homogeneous space.
\end{abstract}

\section{Introduction}\label{sec:intro}

Consider the sphere $\mathbb{S}^d$ embedded into $\mathbb{R}^{d+1}$ as follows: $\mathbb{S}^d=\{\,\mathbf{x}\in\mathbb{R}^{d+1}\colon\|\mathbf{x}\|=1\,\}$, and define the distance between the points $\mathbf{x}_1$ and $\mathbf{x}_2$ by $\rho(\mathbf{x}_1,\mathbf{x}_2)=\cos^{-1}(\mathbf{x}_1^{\top}\mathbf{x}_2)$. With this distance, any isometry between two pairs of points can be extended to an isometry of $\mathbb{S}^d$. A metric space with such a property is called \emph{two-point homogeneous}. A complete classification of \emph{connected and compact} two-point homogeneous spaces is performed in \cite{MR0047345}. Besides spheres, the list includes projective spaces over different algebras; see Section~\ref{sec:2} for details. It turns out that any such space is a \emph{manifold}. We denote it by $\mathbb{M}^d$, where $d$ is the topological dimension of the manifold. Following \cite{MR3736190}, denote by $\mathbb{T}$ either the set $\mathbb{R}$ of real numbers or the set $\mathbb{Z}$ of integers, and call it the \emph{temporal domain}.

Let $(\Omega,\mathfrak{F},\mathsf{P})$ be a probability space.

\begin{definition}
An $\mathbb{R}^m$-valued spatio-temporal random field $\mathbf{Z}(\omega,\mathbf{x},t)\colon\Omega\times\mathbb{M}^d\times\mathbb{T}\to\mathbb{R}^m$ is called (wide-sense) \emph{isotropic} over $\mathbb{M}^d$ and (wide-sense) \emph{stationary} over the temporal domain $\mathbb{T}$, if its mean function $\mathsf{E}[\mathbf{Z}(\mathbf{x}; t)]$ equals a constant vector, and its covariance matrix function
\[
\begin{aligned}
 \cov(\mathbf{Z}(\mathbf{x}_1; t_1), \mathbf{Z}(\mathbf{x}_2; t_2)) & = \mathsf{E}[(\mathbf{Z}(\mathbf{x}_1; t_1)
-\mathsf{E}[\mathbf{Z}(\mathbf{x}_1; t_1)])(\mathbf{Z}(\mathbf{x}_2; t_2)
-\mathsf{E}[\mathbf{Z}(\mathbf{x}_2; t_2)])^{\top}], \\
 &\qquad  \mathbf{x}_1,  \mathbf{x}_2 \in \mathbb{M}^d, t_1,  t_2 \in \mathbb{T},
 \end{aligned}
 \]
  depends only on the time lag $t_2-t_1$ between $t_2$ and $t_1$ and the distance $\rho(\mathbf{x}_1,\mathbf{x}_2)$ between $\mathbf{x}_1$ and $\mathbf{x}_2$.
\end{definition}

As usual, we omit the argument $\omega\in\Omega$ in the notation for the random field under consideration.
In such a case, the covariance matrix function is denoted by $\mathsf{C} ( \rho (\mathbf{x}_1, \mathbf{x}_2); t)$,
\begin{eqnarray*}
 \mathsf{C} (\rho (\mathbf{x}_1, \mathbf{x}_2); t_1-t_2)  & =&
 \mathsf{E}[(\mathbf{Z}(\mathbf{x}_1; t_1)
-\mathsf{E}[\mathbf{Z}(\mathbf{x}_1; t_1)])(\mathbf{Z}(\mathbf{x}_2; t_2)
-\mathsf{E}[\mathbf{Z}(\mathbf{x}_2; t_2)])^{\top}], \\
 &  &   \qquad  \qquad  \mathbf{x}_1,  \mathbf{x}_2 \in \mathbb{M}^d, t_1,  t_2 \in \mathbb{T}.
 \end{eqnarray*}
   It is an $m \times m$ matrix function,  $\mathsf{C} (\rho (\mathbf{x}_1, \mathbf{x}_2); -t) = ( \mathsf{C} (\rho (\mathbf{x}_1, \mathbf{x}_2); t) )^{\top}$, and the inequality
    \[
      \sum_{i=1}^n \sum_{j=1}^n  \mathbf{a}^{\top}_i \mathsf{C} (\rho (\mathbf{x}_i, \mathbf{x}_j); t_i-t_j) \mathbf{a}_j \ge 0
     \]
      holds for every   $n \in \mathbb{N}$,  any $\mathbf{x}_i \in \mathbb{M}^d$, $t_i \in \mathbb{T}$,  and $\mathbf{a}_i \in \mathbb{R}^m$ ($ i =1, 2, \ldots, n$), where $\mathbb{N}$ stands for the set of positive integers, while $\mathbb{N}_0$ denotes the set of nonnegative integers below.
      On the other hand, given an $m \times m$ matrix function with these properties, there exists an $m$-variate Gaussian or elliptically contoured random field   $\{\, \mathbf{Z} (\mathbf{x}; t)\colon \mathbf{x} \in \mathbb{M}^d, t \in \mathbb{T}\, \}$
      with  $\mathsf{C} ( \rho (\mathbf{x}_1, \mathbf{x}_2); t)$ as its covariance matrix function \cite{MR2774237}.

For a scalar and purely spatial  random field $\{\, Z(\mathbf{x})\colon \mathbf{x} \in \mathbb{M}^d\, \}$ that is isotropic and mean square continuous, its covariance function  is continuous and possesses a series representation of the form \cite{MR0005782,MR0215331,MR0005922}
\begin{equation}
      \label{scalar.cov}
      \cov  ( Z (\mathbf{x}_1), Z( \mathbf{x}_2)) =  \sum\limits_{n=0}^\infty  b_n  P_n^{  (\alpha, \beta) } \left( \cos (\rho (\mathbf{x}_1, \mathbf{x}_2)) \right),\qquad\mathbf{x}_1, \mathbf{x}_2 \in \mathbb{M}^d,
      \end{equation}
         where  $\{\, b_n\colon n \in \mathbb{N}_0\, \}$ is a sequence of nonnegative numbers with  $\sum\limits_{n=0}^\infty  b_n  P_n^{  (\alpha, \beta) } (1)$ convergent,
      $P_n^{ (\alpha, \beta)}  (x)$ is a   Jacobi  polynomial of degree $n$ with a pair of parameters $(\alpha, \beta)$ \cite{MR1688958,MR0372517},
     shown in Table~\ref{tab:2} below.
   A general form of the covariance matrix function and
    a series representation  are derived in \cite{MR3736190}  for a vector random field that is isotropic and mean square continuous on a sphere and stationary on a temporal domain. They are extended to $\mathbb{M}^d \times \mathbb{T}$ in this paper.

Isotropic random fields over $\mathbb{S}^d$ with values in $\mathbb{R}^1$ and $\mathbb{C}^1$ were introduced in \cite{obukhov1947statistically}. Theoretical investigations and practical applications of isotropic scalar-valued random fields on spheres may be found in \cite{MR0339308,MR3449810,MR2946126,MR3092051,MR0146880},  and vector- and tensor-valued random fields on spheres have been considered in \cite{MR2870527,MR3488255,MR3736190,MR1213893}, among others. Cosmological applications, in particular, studies of tiny fluctuations of the Cosmic Microwave Background, require development of the theory of \emph{random sections of vector and tensor bundles} over $\mathbb{S}^2$ \cite{MR3170229,MR2737761,MR2884225,MR3746005}. See also surveys of the topic in the monographs \cite{MR2977490,MR2840154,MR697386,MR893393}.
Isotropic random fields on connected compact two-point homogeneous spaces are studied in \cite{MR0423000,MR0215331,MR1712757,MR2110910,MR550252}, among others.

Some important properties of   $\mathbb{M}^d$,  $\rho (\mathbf{x}_1, \mathbf{x}_2)$,    and  $P_n^{(\alpha, \beta)} (x)$ are reviewed in Section~\ref{sec:2} and    two lemmas  are derived, one as a special case of the   Funk--Hecke formula  on $\mathbb{M}^d$ and  the other as a kind of  probability interpretation. A series representation is given in Section~\ref{sec:3} for an  isotropic and mean square continuous vector  random field on  $\mathbb{M}^d$, and a series expression of  its covariance matrix function, in terms of Jacobi polynomials. Section~\ref{sec:4} deals with a spatio-temporal vector random field on $\mathbb{M}^d\times\mathbb{T}$, which is isotropic  and mean square continuous vector  random field on  $\mathbb{M}^d$ and stationary on $\mathbb{T}$, and obtains a series representation for the random field and a general form for its covariance matrix function.
 The lemmas and theorems are proved in Appendix~\ref{sec:proofs}.

\section{Compact Two-Point Homogeneous Spaces and Jacobi Polynomials}\label{sec:2}

This section starts by  recalling some important properties of the compact two-point homogeneous space $\mathbb{M}^d$  and those of Jacobi polynomials, and then establishes   two useful lemmas on a special case of the   Funk--Hecke formula  on $\mathbb{M}^d$ and  its probability interpretation, which are conjectured in \cite{MR3736190}.

The compact connected two-point homogeneous spaces  are shown in  the first column of  Table~\ref{tab:1}. Besides spheres, there are projective spaces over the fields $\mathbb{R}$ and $\mathbb{C}$, over the skew field $\mathbb{H}$ of quaternions, and over the algebra $\mathbb{O}$ of octonions. The possible values of $d$ are chosen in such a way that all the spaces in Table~\ref{tab:1} are different and exhaust the list. In the lowest dimensions we have $\mathbb{P}^1(\mathbb{R})=\mathbb{S}^1$, $\mathbb{P}^2(\mathbb{C})=\mathbb{S}^2$, $\mathbb{P}^4(\mathbb{H})=\mathbb{S}^4$, and $\mathbb{P}^8(\mathbb{O})=\mathbb{S}^8$.

\begin{table}
\caption{An approach based on Lie algebras}
\label{tab:1}       
\begin{tabular}{llllll}
\hline\noalign{\smallskip}
 $\mathbb{M}^d$ & $G$ & $K$ & $p$ & $q$ & Zonal function \\
\noalign{\smallskip}\hline\noalign{\smallskip}
$\mathbb{S}^d$, $d=1$, $2$, \dots & $\mathrm{SO}(d+1)$ & $\mathrm{SO}(d)$ & $0$ & $d-1$ & $R^{(\alpha,\beta)}_{n}(\cos(\rho(\mathbf{x},\mathbf{o})))$ \\
    $ \mathbb{P}^d(\mathbb{R})$, $d=2$, $3$, \dots & $\mathrm{SO}(d+1)$ & $\mathrm{O}(d)$ & $0$ & $d-1$ & $R^{(\alpha,\beta)}_{2n}(\cos(\rho(\mathbf{x},\mathbf{o})/2))$ \\
    $ \mathbb{P}^d(\mathbb{C})$, $d=4$, $6$, \dots & $\mathrm{SU}(\frac{d}{2}+1)$ & $\mathrm{S}(\mathrm{U}(\frac{d}{2})\times\mathrm{U}(1))$ & $d-2$ & $1$ & $R^{(\alpha,\beta)}_{n}(\cos(\rho(\mathbf{x},\mathbf{o})))$ \\
    $ \mathbb{P}^d(\mathbb{H})$, $d=8$, $12$, \dots & $\mathrm{Sp}(\frac{d}{4}+1)$ & $\mathrm{Sp}(\frac{d}{4})\times\mathrm{Sp}(1)$ & $d-4$ & $3$ & $R^{(\alpha,\beta)}_{n}(\cos(\rho(\mathbf{x},\mathbf{o})))$ \\
    $ \mathbb{P}^{16}(\mathbb{O})$ & $\mathrm{F}_{4(-52)}$ & $\mathrm{Spin}(9)$ & $8$ & $7$ & $R^{(\alpha,\beta)}_{n}(\cos(\rho(\mathbf{x},\mathbf{o})))$ \\
\noalign{\smallskip}\hline
\end{tabular}
\end{table}

All compact two-point homogeneous spaces share the same property \cite{MR496885} that all of their geodesic lines are closed. Moreover, all of them are circles and have the same length. In particular, when the sphere $\mathbb{S}^d$ is embedded into the space $\mathbb{R}^{d+1}$ as described in Section~\ref{sec:intro}, the length of any geodesic line is equal to that of the unit circle, that is, $2\pi$. It is natural no norm the distance in such a way that the length of any geodesic line is equal to $2\pi$, exactly as in the case of the unit sphere.

There are at least two different approaches to the subject of compact two-point homogeneous spaces  in the literature. Two of them are reviewed in the next two subsections.

\subsection{An approach based on Lie algebras}\label{sub:Lie}

This approach goes back to  Cartan \cite{MR1509283}. It has been used in both probabilistic literature \cite{MR0215331} and  approximation theory literature \cite{MR3722488}.

Let $G$ be the connected component of the group of isometries of $\mathbb{M}^d$, and let $K$ be the stationary subgroup of a fixed point in $\mathbb{M}^d$, call it $\mathbf{o}$.  Cartan \cite{MR1509283} defined and calculated the numbers $p$ and $q$, which are dimensions of some root spaces connected with the Lie algebras of the groups $G$ and $K$. The groups $G$ and $K$ are listed in the second and the third columns of Table~\ref{tab:1}, while the numbers $p$ and $q$ are listed in the fourth and fifth columns of the table.

By \cite[Theorem~11]{MR0117681}, if $\mathbb{M}^d$ is a two-point homogeneous space, then the only differential operators on $\mathbb{M}^d$ that are invariant under all isometries of $\mathbb{M}^d$ are the polynomials in a special differential operator $\Delta$ called the \emph{Laplace--Beltrami operator}. Let $\mathrm{d}\nu(\mathbf{x})$ be the measure which is induced on the homogeneous space $\mathbb{M}^d=G/K$ by the \emph{probabilistic} invariant measure on $G$. It is possible to define $\Delta$ as a self-adjoint operator in the space $H=L^2(\mathbb{M}^d,\mathrm{d}\nu(\mathbf{x}))$. The spectrum of $\Delta$ is discrete and the eigenvalues are
\[
\lambda_{n}=-\varepsilon n(\varepsilon n+\alpha+\beta+1), ~~~~~~ n \in \mathbb{N}_0,
\]
where
\begin{equation}\label{eq:3}
\alpha=(p+q-1)/2,\qquad\beta=(q-1)/2,
\end{equation}
and where $\varepsilon=2$ if $\mathbb{M}^d= \mathbb{P}^d(\mathbb{R})$ and $\varepsilon=1$ otherwise.

Let $H_{n}$ be the eigenspace of $\Delta$ corresponding to $\lambda_{n}$. The space $H$ is the Hilbert direct sum of its subspaces $H_{n}$, $n\in\mathbb{N}_0$. The space $H_n$ is finite-dimensional with
\[
\dim H_n= \frac{(2n+\alpha+\beta+1)\Gamma(\beta+1)\Gamma(n+\alpha+\beta+1)\Gamma(n+\alpha+1)}
{\Gamma(\alpha+1)\Gamma(\alpha+\beta+2)\Gamma(n+1)\Gamma(n+\beta+1)}.
\]
Each of the spaces $H_{n}$ contains a unique one-dimensional subspace whose elements are \emph{$K$-spherical functions}; that is, functions invariant under the action of $K$ on $\mathbb{M}^d$. Such a function, say $f_{n}(\mathbf{x})$, depends only on the distance $r=\rho(\mathbf{x},\mathbf{o})$, $f_{n}(\mathbf{x})=f^*_{n}(r)$. A spherical function is called \emph{zonal} if $f^*_{n}(0)=1$.

The zonal spherical functions of all compact connected two-point homogeneous spaces are listed in the last column of Table~\ref{tab:1}. To explain notation, we recall that the \emph{Jacobi polynomials}
\[
\begin{aligned}
&P_n^{(\alpha, \beta)} (x) = \frac{\Gamma (\alpha+n+1)}{n! \Gamma (\alpha+\beta+n+1)}\sum_{k=0}^n\binom{n}{k}\frac{\Gamma (\alpha+\beta+n+k+1)}{\Gamma ( \alpha+k+1 )} \left(\frac{x-1}{2} \right)^k,\\ &\quad x \in [-1,1],\quad n \in \mathbb{N}_0,
\end{aligned}
\]
are the eigenfunctions of the \emph{Jacobi operator} \cite[Theorem~4.2.1]{MR0372517}
\[
\Delta_x=\frac{1}{(1-x)^{\alpha}(1+x)^{\beta}}\frac{\mathrm{d}}{\mathrm{d}x}
\left((1-x)^{\alpha+1}(1+x)^{\beta+1}\frac{\mathrm{d}}{\mathrm{d}x}\right).
\]
In the last column of Table~\ref{tab:1},  the \emph{normalised Jacobi polynomials} are introduced,
\[
R^{(\alpha,\beta)}_{n}(x)=\frac{P^{(\alpha,\beta)}_{n}(x)}{P^{(\alpha,\beta)}_{n}(1)},   \qquad n \in \mathbb{N}_0,
\]
where
  \begin{equation}\label{eq:4}
P^{(\alpha,\beta)}_{n}(1)=\frac{\Gamma(n+\alpha+1)}{\Gamma(n+1)\Gamma(\alpha+1)}.
\end{equation}

The reason for the exceptional behaviour of the real projective spaces is as follows;  see \cite{MR0215331,MR1954548}. The space $ \mathbb{P}^d(\mathbb{R})$ may be constructed by identification of antipodal points on the sphere $\mathbb{S}^d$. An $\mathrm{O}(d)$-invariant function $f$ on $ \mathbb{P}^d(\mathbb{R})$ can be lifted to an $\mathrm{SO}(d)$-invariant function $g$ on $\mathbb{S}^d$ by $g(\mathbf{x})=f(\pi(\mathbf{x}))$, where $\pi$ maps a point $\mathbf{x}\in\mathbb{S}^d$ to the pair of antipodal points $\pi(\mathbf{x})\in   \mathbb{P}^d(\mathbb{R})$. This simply means that a function on $[0,1]$ can be extended to an even function on $[-1,1]$. Only the even polynomials can be functions on the so constructed manifold. By \cite[Equation~(4.1.3)]{MR0372517}, we have
\[
P^{(\alpha,\beta)}_{n}(x)=(-1)^{n}P^{(\beta,\alpha)}_{n}(-x).
\]
For the real projective spaces $\alpha=\beta$, and the corresponding normalised Jacobi polynomials are even if and only if $n$ is even.

\begin{remark}
If two Lie groups have the same connected component of identity, then they have the same Lie algebra. For example, the group $\mathrm{SO}(d)$ and $\mathrm{O}(d)$ have the same Lie algebra $\mathfrak{so}(d)$. That is, the approach based on Lie algebras gives the same values of $p$ and $q$ for spheres and real projective spaces of equal dimensions. Only zonal spherical functions can distinguish between the two cases.

In the only case of $\mathbb{M}^d=\mathbb{S}^1$, we have $p=q=0$. The reason is that only in this case the Lie algebra $\mathfrak{so}(2)$ is commutative rather than semisimple, and does not have nonzero root spaces at all.
\end{remark}

\subsection{A geometric approach}

There is a trick that allows us  to write down \emph{all} zonal spherical functions of \emph{all} compact two-point homogeneous spaces in the same form,
 which is used in probabilistic literature   \cite{MR0423000,MR2977490,MR1712757,MR2110910,MR550252} and in approximation theory  \cite{MR2119285,MR3748139}. Denote $y=\cos(\rho(\mathbf{x},\mathbf{o})/2)$. Then we have $\cos(\rho(\mathbf{x},\mathbf{o}))=2y^2-1$. For the case of $\mathbb{M}^d= \mathbb{P}^d(\mathbb{R})$, $\alpha=\beta=(d-2)/2$. By \cite[Theorem~4.1]{MR0372517},
\[
P^{(\alpha,\alpha)}_{2n}(y)=\frac{\Gamma(2n+\alpha+1)\Gamma(n+1)}
{\Gamma(n+\alpha+1)\Gamma(2 n+1)}P^{(\alpha,-1/2)}_{n}(2y^2-1).
\]
In terms of the normalised Jacobi polynomials we obtain
\[
R^{(\alpha,\alpha)}_{2n}(\cos(\rho(\mathbf{x},\mathbf{o})/2))
=R^{(\alpha,-1/2)}_{n}(\cos(\rho(\mathbf{x},\mathbf{o}))).
\]
For the case of $\mathbb{M}^d= \mathbb{P}^d(\mathbb{R})$, if we redefine $\alpha=(d-2)/2$, $\beta=-1/2$, then, \emph{all} zonal spherical functions of \emph{all} compact two-point homogeneous spaces are given by the same expression $R^{(\alpha,\beta)}_{n}(\cos(\rho(\mathbf{x},\mathbf{o})))$.

It easily follows from \eqref{eq:3} that the new values for $p$ and $q$ in the case of $\mathbb{M}^d=P^d(\mathbb{R})$ are $p=d-1$ and $q=0$. It is interesting to note that the new values of $p$ and $q$ for the real projective spaces together with their old values for the rest of spaces still have a meaning;  see \cite{MR3748139} and Table~\ref{tab:2}. This time, the values of $p$ and $q$ are connected with the \emph{geometry} of the space $\mathbb{M}^d$ rather than with Lie algebras.

\begin{table}
\caption{A geometric approach}
\label{tab:2}       
\begin{tabular}{lllllllll}
\hline\noalign{\smallskip}
 $\mathbb{M}^d$ & $p$ & $q$ & $\alpha$ & $\beta$ & $\mathbb{A}$ & $i(\mathbb{M}^d)$ \\
\noalign{\smallskip}\hline\noalign{\smallskip}
$\mathbb{S}^d$, $d=1$, $2$, \dots & $0$ & $d-1$ & $\frac{d-2}{2}$ & $\frac{d-2}{2}$ & $\mathbb{S}^0$ & $1$ \\
    $ \mathbb{P}^d(\mathbb{R})$, $d=2$, $3$, \dots & $d-1$ & $0$ & $\frac{d-2}{2}$ & $-\frac{1}{2}$ & $ \mathbb{P}^{d-1}(\mathbb{R})$ & $2^{d-1}$ \\
    $ \mathbb{P}^d(\mathbb{C})$, $d=4$, $6$, \dots & $d-2$ & $1$ & $\frac{d-2}{2}$ & $0$ & $ \mathbb{P}^{d-2}(\mathbb{C})$ & $\binom{d-1}{d/2-1}$ \\
    $  \mathbb{P}^d(\mathbb{H})$, $d=8$, $12$, \dots & $d-4$ & $3$ & $\frac{d-2}{2}$ & $1$ & $ \mathbb{P}^{d-4}(\mathbb{H})$ & $\frac{1}{d/2+1}\binom{d-1}{d/2-1}$ \\
    $  \mathbb{P}^{16}(\mathbb{O})$ & $8$ & $7$ & $7$ & $3$ & $ \mathbb{P}^{8}(\mathbb{O})$ & $39$ \\
\noalign{\smallskip}\hline
\end{tabular}
\end{table}

Specifically, let $\mathbb{A}=\{\,\mathbf{x}\in\mathbb{M}^d\colon\rho(\mathbf{x},\mathbf{o})=\pi\,\}$. This set is called the \emph{antipodal manifold} of the point $\mathbf{o}$. The antipodal manifolds are listed in the last column of Table~\ref{tab:2}. Geometrically, if $\mathbb{M}^d=\mathbb{S}^d$ and $\mathbf{o}$ is the North pole, then $\mathbb{A}=\mathbb{S}^0$ is the South pole. Otherwise, $\mathbb{A}$ is the \emph{space at infinity} of the point $\mathbf{o}$ in the terms of projective geometry. The new number $p$ turns out to be the \emph{dimension of the antipodal manifold}, while the number $p+q+1$ is, as before, the dimension of the space $\mathbb{M}^d$ itself.

In what follows, we use the geometric approach.
It turns out that all the spaces $\mathbb{M}^d$ are \emph{Riemannian manifolds}, as is defined in \cite{MR0282313}. Each Riemannian manifold carries the \emph{canonical measure} $\mu$; see \cite[p.~10--11]{MR0282313}. The measure $\mu$ is proportional to the measure $\nu$ constructed in Subsection~\ref{sub:Lie}. The coefficient of proportionality, or the total measure $\mu(\mathbb{M}^d)$ of the compact manifold $\mathbb{M}^d$ is called the \emph{volume} of $\mathbb{M}^d$.

\begin{lemma}\label{lem:3}
The volume of the space $\mathbb{M}^d$ is
\begin{equation}\label{eq:7}
\omega_d=\mu(\mathbb{M}^d)=\frac{(4\pi)^{\alpha+1}\Gamma(\beta+1)}{\Gamma(\alpha+\beta+2)}.
\end{equation}
\end{lemma}

In what follows, we write just $\mathrm{d}\mathbf{x}$ instead of $\mathrm{d}\mu(\mathbf{x})$.

\subsection{Orthogonality properties of Jacobi polynomials}

The set of Jacobi polynomials $\{\, P_n^{(\alpha, \beta)} (x)\colon n \in \mathbb{N}_0, x \in \mathbb{R}\, \}$
         possesses two types of  orthogonality  properties. First, for each pair of $\alpha>-1$ and $\beta>-1$, this set is a complete orthogonal system on  the interval $[-1, 1]$ with respect to the weight
         function $(1-x)^\alpha (1+x)^\beta$, in the sense that
           \begin{equation}
            \label{Gegenbauer.orth}
             \int_{-1}^1 P^{(\alpha, \beta)}_i (x) P^{(\alpha, \beta)}_j (x)  (1-x)^\alpha (1+x)^\beta  d x
                =  \left\{
                \begin{array}{ll}
                \frac{2^{\alpha+\beta+1} }{2 j +\alpha+\beta+1} \frac{\Gamma (j+\alpha+1) \Gamma (j+\beta+1)}{ j! \Gamma ( j +\alpha+\beta+1) },
                 ~   &  ~ i =j, \\
                 0, ~ & ~ i \neq j.
                 \end{array}   \right.
                \end{equation}

Second, for \emph{selected values} of $\alpha$ and $\beta$ given by \eqref{eq:3} with $p$ and $q$ given in Table~\ref{tab:2}, they are  orthogonal  over     $\mathbb{M}^d$, as the following lemma describes, which
              is derived from the   Funk--Hecke formula  recently established in \cite{MR3722488}.  In the particular case $\mathbb{M}^d=\mathbb{S}^d$,  the   Funk--Hecke formula  may be found in classical references such as \cite{MR1688958,MR1483320}.

\begin{lemma}\label{lem:1}
For $i, j \in \mathbb{N}_0$,  and $\mathbf{x}_1$, $\mathbf{x}_2 \in \mathbb{M}^d$,
\[
 \int_{\mathbb{M}^d } P_i^{(\alpha,\beta) } (\cos(\rho(\mathbf{x}_1,\mathbf{x})))
P_j^{(\alpha,\beta)} (\cos(\rho(\mathbf{x}_2,\mathbf{x})))\,\mathrm{d}\mathbf{x}
 =
      \frac{\delta_{ij}\omega_d}{a_i^2} P_i^{(\alpha,\beta)} (\cos(\rho(\mathbf{x}_1,\mathbf{x}_2))),
      \]
where
\begin{equation}
\label{an.def}
a_n=\left(\frac{\Gamma(\beta+1)(2 n +\alpha+\beta+1)\Gamma(n+\alpha+\beta+1)}
{\Gamma(\alpha+\beta+2)\Gamma(n+\beta+1)}\right)^{\frac{1}{2}},\qquad n \in \mathbb{N}_0.
\end{equation}
\end{lemma}

The probabilistic interpretation of zonal spherical functions on $\mathbb{M}^d$ is provided in Lemma~\ref{lem:2} below. The spherical case is given in \cite{MR3488255}.

\begin{definition}
A random vector $\mathbf{U}$   is said to be \emph{uniformly distributed} on  $\mathbb{M}^d$ if, for every Borel set $A\subseteq\mathbb{M}^d$ and every isometry $g$ we have $\mathsf{P} (\mathbf{U}\in A ) =\mathsf{P} (\mathbf{U}\in gA)$.
\end{definition}

To construct $\mathbf{U}$, we start with a measure $\sigma$ proportional to the invariant measure $\nu$ of Subsection~\ref{sub:Lie}. Let $T_{\mathbf{o}}$ be the tangent space to $\mathbb{M}^d$ at the point $\mathbf{o}$. Choose a Cartesian coordinate system in $T_{\mathbf{o}}$ and identify this space with the space $\mathbb{R}^{d+1}$. Construct a chart $\varphi\colon\mathbb{M}^d\setminus\mathbb{A}\to\mathbb{R}^{d+1}$ as follows. Put $\varphi(\mathbf{o})=\mathbf{0}\in\mathbb{R}^d$. For any other point $\mathbf{x}\in\mathbb{M}^d\setminus\mathbb{A}$, draw the unique geodesic line connecting $\mathbf{o}$ and $\mathbf{x}$. Let $\mathbf{r}\in\mathbb{R}^{d+1}$ be the unit tangent vector to the above geodesic line. Define
\[
\varphi(\mathbf{x})= \mathbf{r} \tan(\rho(\mathbf{x},\mathbf{o})/2),
\]
and, for each Borel set $B\subseteq\mathbb{M}^d$,
\[
\sigma(B)=\int_{\varphi^{-1}(B\setminus\mathbb{A})}\frac{\mathrm{d}\mathbf{x}}
{(1+\|\mathbf{x}\|^2)^{\alpha+\beta+2}}.
\]
This measure is indeed invariant \cite[p.~113]{MR3024377}.
Finally, define a probability space $(\Omega', $  $ \mathfrak{F}', $  $\mathsf{P}')$ as follows: $\Omega'=\mathbb{M}^d$, $\mathfrak{F}'$ is the $\sigma$-field of Borel subsets of $\Omega'$, and
\[
\mathsf{P}'(B)=\frac{\sigma(B)}{\sigma(\mathbb{M}^d)},\qquad B\in\mathfrak{B}'.
\]
The random variable $\mathbf{U}(\omega)=\omega$ is then uniformly distributed  on $\mathbb{M}^d$.

\begin{lemma}\label{lem:2}
Let $\mathbf{U}$ be a random vector uniformly distributed on $\mathbb{M}^d$. For $n \in \mathbb{N}$,
\[
Z_n(\mathbf{x})=a_n P_n^{(\alpha,\beta)} (\cos(\rho(\mathbf{x},\mathbf{U}))),
\qquad\mathbf{x}\in\mathbb{M}^d,
\]
is a centred isotropic random field with covariance function
  $$ \cov ( Z_n (\mathbf{x}_1), Z_n (\mathbf{x}_2) )
=P_n^{ (\alpha,\beta)} (\cos(\rho(\mathbf{x}_1, \mathbf{x}_2))),  ~~~~~ \mathbf{x}_1, \mathbf{x}_2 \in \mathbb{M}^d,
$$
where $a_n$ is given by \emph{(\ref{an.def})}.
Moreover, for $k \neq n$,  the random fields $ \{\, Z_k (\mathbf{x})\colon \mathbf{x} \in \mathbb{M}^d\, \}$ and $ \{\,Z_n(\mathbf{x}), \mathbf{x} \in \mathbb{M}^d\, \}$ are uncorrelated:
\begin{equation}
\label{eq:6}
 \cov (Z_k (\mathbf{x}_1), Z_n (\mathbf{x}_2) ) =0,    ~~~~ \mathbf{x}_1, \mathbf{x}_2  \in \mathbb{M}^d.
\end{equation}
\end{lemma}

\section{Isotropic Vector Random Fields on $\mathbb{M}^d$}\label{sec:3}

In the purely spatial case, this section presents a series representation for  an  $m$-variate  isotropic and mean square continuous random field  $\{\, \mathbf{Z} (\mathbf{x})\colon \mathbf{x} \in \mathbb{M}^d\, \}$,
   and a series expression for its covariance matrix function, in terms of Jacobi polynomials.
   By mean square continuous, we mean that, for $k =1, \ldots, m$,
     $$ \mathsf{E}[ | Z_k (\mathbf{x}_1) -Z_k (\mathbf{x}_2) |^2] \to 0,  ~~ \mbox{as} ~~ \rho (\mathbf{x}_1, \mathbf{x}_2 ) \to 0,  ~ \mathbf{x}_1, \mathbf{x}_2 \in \mathbb{M}^d. $$
     It implies the continuity of each entry of the  associated covariance matrix function  in terms of $\rho (\mathbf{x}_1, \mathbf{x}_2)$.

In what follows, $d$ is assumed to be greater than $1$, while   $\mathbb{M}^d$ reduces to the unit circle $\mathbb{S}^1$ when $d=1$, over which the treatment of isotropic vector random fields may be found in
 \cite{MR3488255,MR3736190}.
For an $m \times m$ symmetric and nonnegative-definite matrix $\mathsf{B}$  with  nonnegative eigenvalues $\lambda_1, \dots$, $\lambda_m$, there is an orthogonal matrix $\mathsf{S}$ such that $\mathsf{S}^{-1}\mathsf{B}\mathsf{S}=\mathsf{D}$, where $\mathsf{D}$ is a diagonal matrix with diagonal entries $\lambda_1, \ldots,  \lambda_m$. Define the square-root of $ \mathsf{B}$ by
\[
\mathsf{B}^{\frac{1}{2}}=\mathsf{S}\mathsf{D}^{\frac{1}{2}}\mathsf{S}^{-1},
\]
where $\mathsf{D}^{\frac{1}{2}}$ is a diagonal matrix with diagonal entries $\sqrt{\lambda_1}, \ldots,  \sqrt{ \lambda_m}$. Clearly, $\mathsf{B}^{\frac{1}{2}}$ is symmetric, nonnegative-definite, and $(\mathsf{B}^{\frac{1}{2}})^2=\mathsf{B}$. Denote by $\mathsf{I}_m$ an
   $m \times m$ identity matrix.
   For a sequence of $m \times m$ matrices $\{\, \mathsf{B}_n\colon n \in \mathbb{N}_0 \,\}$, the series $\sum\limits_{n=0}^\infty \mathsf{B}_n$ is said to be convergent, if
   each of its entries is convergent.

\begin{theorem}\label{th:1}
Suppose that
 $\{\, \mathbf{V}_n\colon n  \in \mathbb{N}_0\, \}$ is a sequence of  independent $m$-variate random vectors with
$\mathsf{E} ( \mathbf{V}_n)= \mathbf{0}$ and  $\cov ( \mathbf{V}_n, \mathbf{V}_n ) =  a_n^2\mathsf{I}_m$,     $\mathbf{U}$  is a
        random vector uniformly distributed on $\mathbb{M}^d$  and is independent of
        $\{\, \mathbf{V}_n\colon n  \in \mathbb{N}_0\, \}$, and that $\{\, \mathsf{B}_n\colon n \in \mathbb{N}_0\, \}$ is a sequence of $m \times m$ symmeteric nonnegative-definite matrices.  If  the series  $\sum\limits_{n=0}^\infty \mathsf{B}_n  P_n^{ (\alpha, \beta) } (1)$ converges,  then
       \begin{equation}
     \label{spatiostoc1}
      \mathbf{Z} (\mathbf{x}) = \sum_{n=0}^\infty  \mathsf{B}_n^{\frac{1}{2}} \mathbf{V}_n P_n^{ (\alpha, \beta) } ( \cos \rho (\mathbf{x},  \mathbf{U} )),
      ~~~~~~ \mathbf{x}   \in \mathbb{M}^d,
      \end{equation}
    is a centred $m$-variate isotropic random field on $\mathbb{M}^d$,  with  covariance matrix function
       \begin{equation}
     \label{spatio.cmf1}
      \cov ( \mathbf{Z} (\mathbf{x}_1), \mathbf{Z}(\mathbf{x}_2) )  = \sum_{n=0}^\infty  \mathsf{B}_n  P_n^{(\alpha, \beta) } \left( \cos  \rho (\mathbf{x}_1, \mathbf{x}_2) \right),
      ~~~~~~ \mathbf{x}_1, \mathbf{x}_2 \in \mathbb{M}^d.
      \end{equation}
The terms of \emph{(\ref{spatiostoc1})} are uncorrelated; more precisely,
           $$ \cov  \left(  \mathsf{B}_i^{\frac{1}{2}} \mathbf{V}_i P_i^{  (\alpha, \beta)  } ( \rho (\mathbf{x}_1,  \mathbf{U})),  ~
                         \mathsf{B}_j^{\frac{1}{2}} \mathbf{V}_j P_j^{ (\alpha, \beta)  } ( \rho (\mathbf{x}_2,  \mathbf{U} ))  \right) = \mathbf{0}, ~~~ \mathbf{x}_1, \mathbf{x}_2 \in \mathbb{M}^d, ~ i \neq j. $$
\end{theorem}

Since
   $  \left|  P_n^{ (\alpha, \beta) } (\cos \vartheta) \right|  \le P_n^{  (\alpha, \beta) } (1),  n \in \mathbb{N}_0, $
       the convergent  assumption  of the series $\sum\limits_{n=0}^\infty \mathsf{B}_n P_n^{ (\alpha, \beta) } (1)$
       ensures  not only the mean square convergence of the series at the right hand of (\ref{spatiostoc1}), but also the uniform and absolute convergence of
       the series at the right hand side of (\ref{spatio.cmf1}).

When $\mathbb{M}^d=\mathbb{S}^2$ and $m=1$, we have $\dim H_n=2n+1$, and  \eqref{spatio.cmf1} takes the form
\[
\cov ( Z (\mathbf{x}_1), Z(\mathbf{x}_2) )  = \sum_{n=0}^\infty  b_n  P_n\left( \cos  \rho (\mathbf{x}_1, \mathbf{x}_2) \right),
\]
where $P_n (x) $ are Legendre polynomials. In the theory of Cosmic Microwave Background, this equation is traditionally written in the form
\[
\cov ( Z (\mathbf{x}_1), Z(\mathbf{x}_2) )  = \sum_{\ell=0}^\infty(2\ell+1)C_{\ell}  P_{\ell}\left(\mathbf{x}_1\cdot\mathbf{x}_2\right),
\]
and the sequence $\{\,C_{\ell}\colon\ell\geq 0\,\}$ is called the \emph{angular power spectrum}. In the general case, define the angular power spectrum by
\[
\mathsf{C}_n=\frac{1}{\dim H_n}\mathsf{B}_n.
\]
A lot of examples of the angular power spectrum for general compact two-point homogeneous spaces may be found in \cite{MR0423000}.

As the next theorem indicates,  (\ref{spatio.cmf1}) is a general form that the covariance matrix function of an $m$-variate isotropic and mean square continuous random field on $\mathbb{M}^d$
      must take.

\begin{theorem}\label{th:2}
For an $m$-variate isotropic and mean square continuous random field $ \{\, Z(\mathbf{x})\colon \mathbf{x} \in \mathbb{M}^d\, \}$, its covariance matrix function
$\cov ( Z(\mathbf{x}_1),  Z (\mathbf{x}_2) ) $ is of the form
    \begin{equation}
     \label{spatio.cmf2}
      \mathsf{C} ( \mathbf{x}_1, \mathbf{x}_2 ) = \sum_{n=0}^\infty  \mathsf{B}_n  P_n^{ (\alpha, \beta) } \left( \cos  \rho (\mathbf{x}_1, \mathbf{x}_2)  \right),
      ~~~~~~  \mathbf{x}_1, \mathbf{x}_2 \in \mathbb{M}^d,
      \end{equation}
where
      $\{\,\mathsf{B}_n\colon n \in \mathbb{N}_0\, \}$ is a sequence of $m \times m$ nonnegative-definite matrices and the series  $\sum\limits_{n=0}^\infty \mathsf{B}_n  P_n^{ (\alpha, \beta) } (1)$ converges.

      Conversely, if an $m \times m $ matrix function  $\mathsf{C} (\mathbf{x}_1, \mathbf{x}_2)$
is of the form \emph{(\ref{spatio.cmf2})},  then  it is the covariance matrix function   of an $m$-variate isotropic Gaussian or elliptically contoured random field on $\mathbb{M}^d$.
\end{theorem}

 Examples of covariance matrix functions on $\mathbb{S}^d$ may be found in, for instance,  \cite{MR3488255,MR3736190}.
We would call for parametric and semi-parametric covariance matrix structures on $\mathbb{M}^d$.

\section{Time Varying Isotropic Vector Random Fields on $\mathbb{M}^d$}\label{sec:4}

For  an  $m$-variate random field  $\{\, \mathbf{Z} (\mathbf{x}; t)\colon \mathbf{x} \in \mathbb{M}^d, t \in \mathbb{T}\, \}$ that  is
 isotropic and mean square continuous  over  $\mathbb{M}^d$ and stationary on $\mathbb{T}$,
 this section presents the general form of   its covariance matrix function $\mathsf{C} (\rho (\mathbf{x}_1, \mathbf{x}_2); t)$,
  which  is a continuous  function of $\rho (\mathbf{x}_1, \mathbf{x}_2)$ and is also a continuous function of $t \in \mathbb{R}$ if $\mathbb{T} = \mathbb{R}$.
  A series representation is given in the following theorem for such a random field, as an extension of that on $\mathbb{S}^d \times \mathbb{T}$.

\begin{theorem}\label{th:3}
If  an  $m$-variate random field  $\{ \mathbf{Z} (\mathbf{x}; t), \mathbf{x} \in \mathbb{M}^d, t \in \mathbb{T} \}$ is     isotropic and mean square continuous  over  $\mathbb{M}^d$ and stationary on $\mathbb{T}$, then
\[
\mathsf{C} (\rho (\mathbf{x}_1, \mathbf{x}_2); -t) = ( \mathsf{C} (\rho (\mathbf{x}_1, \mathbf{x}_2); t) )^{\top},
\]
 and
     $ \frac{\mathsf{C} (\rho (\mathbf{x}_1, \mathbf{x}_2); t) + \mathsf{C} (\rho (\mathbf{x}_1, \mathbf{x}_2); -t)}{2} $ is of the form
       \begin{equation}
       \label{cov.matrix.fun.1}
        \frac{\mathsf{C} (\rho (\mathbf{x}_1, \mathbf{x}_2); t) + \mathsf{C} (\rho (\mathbf{x}_1, \mathbf{x}_2); -t)}{2}   =
          \sum\limits_{n=0}^\infty  \mathsf{B}_n (t) P_n^{ (\alpha, \beta) } (\cos \rho (\mathbf{x}_1, \mathbf{x}_2)),
       \end{equation}
    \hfill  $\mathbf{x}_1$, $\mathbf{x}_2\in\mathbb{M}^d$, $t\in\mathbb{T}$,

    \noindent
    where, for each fixed $n \in \mathbb{N}_0$,  $ \mathsf{B}_n (t)$ is a stationary covariance matrix function on $\mathbb{T}$, and,   for each fixed $t \in \mathbb{T}$,   $ \mathsf{B}_n (t)$ ($ n \in \mathbb{N}_0$) are $m \times m $ symmetric matrices and
    $\sum\limits_{n=0}^\infty  \mathsf{B}_n (t)  P_n^{ (\alpha, \beta) } (1)$  converges.
\end{theorem}

While a general form  of  $ \frac{\mathsf{C} ( \rho (\mathbf{x}_1, \mathbf{x}_2); t) + \mathsf{C} (\rho (\mathbf{x}_1, \mathbf{x}_2); -t)}{2} $, instead of $\mathsf{C} (\rho (\mathbf{x}_1, \mathbf{x}_2); t)$ itself,  is given in Theorem~\ref{th:3},
    that of  $\mathsf{C} (\rho (\mathbf{x}_1, \mathbf{x}_2); t)$ can be obtained in certain special cases, such as  spatio-temporal symmetric, and purely spatial.

\begin{corollary}
If $\mathsf{C} (\rho (\mathbf{x}_1, \mathbf{x}_2); t)$ is spatio-temporal symmetric in the sense  that
               $$ \mathsf{C} ( \rho (\mathbf{x}_1, \mathbf{x}_2);  - t  ) =\mathsf{C} ( \rho (\mathbf{x}_1, \mathbf{x}_2);  t  ),   ~~~~~~~~ \mathbf{x}_1, \mathbf{x}_2 \in \mathbb{M}^d, ~ t \in \mathbb{T}, $$
          then   it  takes  the form
     $$   \mathsf{C} (\rho (\mathbf{x}_1, \mathbf{x}_2); t)    =
          \sum\limits_{n=0}^\infty  \mathsf{B}_n (t) P_n^{ (\alpha, \beta) } (\cos \rho (\mathbf{x}_1, \mathbf{x}_2)),   ~~  \mathbf{x}_1, \mathbf{x}_2 \in \mathbb{M}^d,    ~ t \in \mathbb{T}.
            $$
\end{corollary}

In contrast to those in (\ref{cov.matrix.fun.1}),  the $m \times m$ matrices  $ \mathsf{B}_n (t)$ ($ n \in \mathbb{N}_0$) in the next theorem are not necessarily   symmetric.
       One simple such  example is
            $$  \mathsf{B} (t) = \left\{
                                           \begin{array}{ll}
                                           \mathsf{\Sigma}+\mathsf{\Phi}        \mathsf{\Sigma}  \mathsf{\Phi}^{\top},  ~  &  ~ t =0, \\
                                            \mathsf{\Phi}        \mathsf{\Sigma},   ~   &  ~  t = -1, \\
                                            \mathsf{\Sigma}  \mathsf{\Phi}^{\top},  ~  &  ~ t = 1, \\
                                            \mathsf{0},  ~ & ~  t = \pm 2, \pm 3, \ldots,
                                            \end{array}    \right. $$
    which is     the covariance matrix function of  an $m$-variate first order moving average time series $\mathbf{Z} (t) = \bm{\varepsilon}  (t) + \mathsf{\Phi}  \bm{\varepsilon} (t-1),
       t \in \mathbb{Z}, $  where $\{\,  \bm{\varepsilon}  (t)\colon t \in \mathbb{Z}\, \}$ is $m$-variate  white noise with $\mathsf{E}[ \bm{\varepsilon}  (t)] = \mathbf{0}$ and
       $\var [\bm{\varepsilon}  (t)] = \mathsf{\Sigma}$, and $\mathsf{\Phi}$ is an $m \times m$ matrix.

\begin{theorem}\label{th:4}
 An  $m \times m$ matrix function
         \begin{equation}
             \label{cov.matrix.fun.2}
        \mathsf{C} ( \rho (\mathbf{x}_1, \mathbf{x}_2); t)    =
          \sum\limits_{n=0}^\infty  \mathsf{B}_n (t)  P_n^{ (\alpha, \beta) } (\cos \rho (\mathbf{x}_1, \mathbf{x}_2)), ~   ~~   \mathbf{x}_1, \mathbf{x}_2 \in \mathbb{M}^d, ~ t \in \mathbb{T},
          \end{equation}
       is the covariance matrix function of   an  $m$-variate Gaussian or elliptically contoured random field
        on $ \mathbb{M}^d \times  \mathbb{T} $  if and only if  $\{\, \mathsf{B}_n (t)\colon n \in \mathbb{N}_0\, \}$ is a sequence of  stationary covariance matrix functions on $\mathbb{T}$
        and $\sum\limits_{n=0}^\infty  \mathsf{B}_n (0) P_n^{ (\alpha, \beta) } (1)$   converges.
\end{theorem}

As an example of (\ref{cov.matrix.fun.2}), let
   $$  \mathsf{B}_n (t) = \left\{
                                           \begin{array}{ll}
                                           \mathsf{\Sigma}_n+\mathsf{\Phi}        \mathsf{\Sigma}_n  \mathsf{\Phi}^{\top},  ~  &  ~ t =0, \\
                                            \mathsf{\Phi}        \mathsf{\Sigma}_n,   ~   &  ~  t = -1, \\
                                            \mathsf{\Sigma}_n  \mathsf{\Phi}^{\top},  ~  &  ~ t = 1, \\
                                            \mathsf{0},  ~ & ~  t = \pm 2, \pm 3, \ldots,  ~~ n \in \mathbb{N}_0,
                                            \end{array}    \right. $$
                                            where $\{ \mathsf{\Sigma}_n: n \in \mathbb{N}_0 \}$ is a sequence
                                            of $m \times m$ nonnegative definite matrices and
                                            $\sum\limits_{n=0}^\infty  \mathsf{\Sigma}_n $ $\times P_n^{ (\alpha, \beta) } (1)$   converges. In this case, (\ref{cov.matrix.fun.2}) is the covariance matrix function of   an  $m$-variate Gaussian or elliptically contoured random field
        on $ \mathbb{M}^d \times  \mathbb{Z}$.

Gaussian and second-order elliptically contoured  random fields  form one of the largest sets, if not the largest set,
          which allows any possible correlation structure \cite{MR2774237}.
          The covariance matrix functions developed in Theorem~\ref{th:4}  can be adopted
          for a Gaussian or elliptically contoured vector random field.
          However, they may not be available for other non-Gaussian  random fields, such as  a log-Gaussian \cite{10.1007/978-94-015-6844-9_2}, $\chi^2$ \cite{5957245}, K-distributed \cite{MR3041387}, or skew-Gaussian one,  for which  admissible correlation structure must be investigated on a case-by-case basis.
          A series representation is given in the following theorem for an $m$-variate spatio-temporal random field on $\mathbb{M}^d\times\mathbb{T}$.

\begin{theorem}\label{th:5}
 An  $m$-variate  random field
       \begin{equation}
     \label{stoc1}
      \mathbf{Z} (\mathbf{x}; t) = \sum_{n=0}^\infty   \mathbf{V}_n (t)  P_n^{  (\alpha, \beta) } ( \cos \rho (\mathbf{x},  \mathbf{U})),
      ~~~~~~ \mathbf{x}   \in \mathbb{M}^d, ~ t \in \mathbb{T},
      \end{equation}
    is isotropic and mean square continuous on $\mathbb{M}^d$,  stationary on $\mathbb{T}$,  and possesses   mean   $\mathbf{0}$ and  covariance matrix function
 \emph{(\ref{cov.matrix.fun.2})}, where
 $\{ \,\mathbf{V}_n (t)\colon  n \in \mathbb{N}_0 \,  \}$ is a sequence of independent $m$-variate  stationary stochastic processes  on $\mathbb{T}$  with
$$ \mathsf{E} ( \mathbf{V}_n )= \mathbf{0},  ~~~  \cov ( \mathbf{V}_n (t_1), \mathbf{V}_n (t_2) ) =  a_n^2 \mathsf{B}_n (t_1-t_2),   ~~~ n \in \mathbb{N}_0, $$
 the   random vector  $\mathbf{U}$  is  uniformly distributed on $\mathbb{M}^d$ and is independent with    $\{\, \mathbf{V}_n (t) \colon  $  $ n \in \mathbb{N}_0\,  \}$,  and
        $\sum\limits_{n=0}^\infty \mathsf{B}_n (0)  P_n^{ (\alpha, \beta) } (1)$ converges.
\end{theorem}

The distinct terms of (\ref{stoc1}) are uncorrelated each other,
           $$ \cov  \left(   \mathbf{V}_i (t) P_i^{ (\alpha, \beta) } ( \cos \rho (\mathbf{x},  \mathbf{U}) ),  ~
                          \mathbf{V}_j  (t) P_j^{ (\alpha, \beta) } ( \cos \rho (\mathbf{x},  \mathbf{U}) )  \right) = \mathbf{0}, ~~~ \mathbf{x} \in \mathbb{M}^d, ~ t \in \mathbb{T},  i \neq j, $$
                           due to Lemma~\ref{lem:2} and the independent assumption among $\mathbf{U}, \mathbf{V}_i (t), \mathbf{V}_j (t)$. The vector stochastic process  $\mathbf{V}_n (t)$ can be expressed as, in terms of $\mathbf{Z} (\mathbf{x}; t)$ and $\mathbf{U}$,
                           $$  \mathbf{V}_n (t) =   \frac{a^2_n}{\omega_d P_n^{ (\alpha, \beta) } (1)}
             \int_{\mathbb{M}^d}  \mathbf{Z} (\mathbf{x}; t) P_n^{ (\alpha, \beta) } (\cos \rho (\mathbf{x},  \mathbf{U})) d \mathbf{x},  ~~~~~ t \in \mathbb{T}, ~  n \in \mathbb{N}_0, $$
where the integral  is understood as a Bochner integral of a function taking values in the Hilbert space of random vectors $\mathbf{Z}\in\mathbb{R}^m$ with $\mathsf{E}[\|\mathbf{Z}\|^2_{\mathbb{R}^m}]<\infty$.

 It is obtained after  we multiply both sides of (\ref{stoc1}) by $P_n^{ (\alpha, \beta) } (\cos \rho (\mathbf{x},  \mathbf{U}))$,   integrate  over $\mathbb{M}^d$, and  apply Lemma~\ref{lem:2},
\[
\begin{aligned}
&\int_{\mathbb{M}^d}  \mathbf{Z} (\mathbf{x}; t) P_n^{ (\alpha, \beta) } (\cos \rho (\mathbf{x},  \mathbf{U})) d \mathbf{x}\\
&\quad= \sum_{k=0}^\infty  \mathbf{V}_n  (t)  \int_{\mathbb{M}^d} P_k^{ (\alpha, \beta) } ( \cos \rho (\mathbf{x},  \mathbf{U}) )
                       P_n^{(\alpha, \beta) } ( \cos \rho (\mathbf{x},  \mathbf{U})) d \mathbf{x}  \\
     &\quad=  \frac{1}{a_n^2}     P_n^{ (\alpha, \beta) } (1) \mathbf{V}_n (t).
\end{aligned}
\]

\appendix
{\scriptsize

\section{Proofs}\label{sec:proofs}

\begin{proof}[Proof of Lemma~\emph{\ref{lem:3}}]
To calculate $\mu(\mathbb{M}^d)$, we use the result of \cite{MR0390968}. If all the geodesics on a $d$-di\-men\-si\-o\-nal Riemannian manifold $M$ are closed and have length $2\pi L$, then the ratio
\[
i(M)=\frac{\mu(\mathbb{M}^d)}{L^n\mu(\mathbb{S}^d)}
\]
is an integer. With our convention $L=1$, we obtain $\mu(\mathbb{M}^d)=i(\mathbb{M}^d)\mu(\mathbb{S}^d)$. It is well-known that
\begin{equation}\label{eq:2}
\mu(\mathbb{S}^d)=\frac{2\pi^{(d+1)/2}}{\Gamma((d+1)/2)}
=\frac{2\pi^{\alpha+3/2}}{\Gamma(\alpha+3/2)}.
\end{equation}
The \emph{Weinstein's integers} $i(\mathbb{M}^d)$ are shown in the last column of Table~\ref{tab:2}. Following \cite{MR776403}, consider all the geodesics from $\mathbf{o}$ to a point in $\mathbb{A}$. Draw a tangent line to each of them and denote by $e$ the dimension of the linear space generated by these lines. We have $e=d$ for $\mathbb{S}^d$, $1$ for $P^d(\mathbb{R})$, $2$ for $P^d(\mathbb{C})$, $4$ for $P^d(\mathbb{H})$, and $8$ for $P^2(\mathbb{O})$. It is proved in \cite{MR776403} that
\[
i(\mathbb{M}^d)=\frac{2^{d-1}\Gamma((d+1)/2)\Gamma(e/2)}{\sqrt{\pi}\Gamma((d+e)/2)}
\]
We know that $d=2\alpha+2$. It is easy to check that $e=2\beta+2$, then we obtain
\[
i(\mathbb{M}^d)=\frac{2^{2\alpha+1}\Gamma(\alpha+3/2)\Gamma(\beta+1)}{\sqrt{\pi}\Gamma(\alpha+\beta+2)},
\]
and \eqref{eq:7} easily follows.
\end{proof}

\begin{proof}[Proof of Lemma~\ref{lem:1}]
In Theorem 2.1 of \cite{MR3722488}  put
\[
K(x)=P_i^{ (\alpha,\beta)} (x),\qquad S(\mathbf{x})=P_j^{(\alpha,\beta)} (\cos(\rho(\mathbf{x}_2,\mathbf{x}))). 
\]
We obtain
\begin{eqnarray*}
& & \int_{\mathbb{M}^d}P_i^{ (\alpha,\beta)} (\cos(\rho(\mathbf{x}_1,\mathbf{x})))
P_j^{ (\alpha,\beta)} (\cos(\rho(\mathbf{x}_2,\mathbf{x})))\,\mathrm{d}\mathbf{x} \\
& = & \omega_d  P_j ^{(\alpha,\beta)} (\cos(\rho(\mathbf{x}_1,\mathbf{x}_2))) \int_{-1}^{1} \frac{P_i^{(\alpha,\beta)} (x)}{ P_i^{(\alpha,\beta)} (1)} P_j^{(\alpha,\beta)} (x)
\mathrm{d}\nu_{\alpha,\beta}(x)  \\
& = & \omega_d  \frac{\delta_{ij}}{a_i^2} P_i^{(\alpha,\beta)} (\cos(\rho(\mathbf{x}_1,\mathbf{x}_2))),
\end{eqnarray*}
where the last equality follows from \eqref{eq:4} , (\ref{Gegenbauer.orth}), and the following well-known result: the probabilistic measure $\nu_{\alpha,\beta}$ on $[-1,1]$, proportional to $(1-x)^{\alpha}(1+x)^{\beta}\,\mathrm{d}x$, is
\begin{equation}\label{eq:1}
\mathrm{d}\nu_{\alpha,\beta}(x)=\frac{\Gamma(\alpha+\beta+2)}{2^{\alpha+\beta+1}
\Gamma(\alpha+1)\Gamma(\beta+1)}(1-x)^{\alpha}(1+x)^{\beta}\,\mathrm{d}x.
\end{equation}

\end{proof}

\begin{proof}[Proof of Lemma~\ref{lem:2}]
The mean function of $\{\, Z_n (\mathbf{x})\colon \mathbf{x} \in \mathbb{M}^d\, \}$ is obtained by  applying
  of \cite[Theorem 2.1]{MR3722488}  to $K(x)=1$ and $S(\mathbf{x})=P^{(\alpha,\beta)}_n (\cos(\rho(\mathbf{x},\mathbf{y})))$,
\[
    \mathsf{E}[Z_n (\mathbf{x})] =  a_n\omega_d \int_{\mathbb{M}^d} P_n^{(\alpha,\beta)} (\cos(\rho(\mathbf{x},\mathbf{y})))
\,\mathrm{d}\mathbf{y} = a_n \cdot 0  =  0.
\]
The covariance function is
 \begin{eqnarray*}
\cov ( Z_n (\mathbf{x}_1), Z_n(\mathbf{x}_2) )
& = & \omega_d^{-1}a_n^2\int_{\mathbb{M}^d} P_n^{ (\alpha,\beta)} (\cos(\rho ( \mathbf{x}_1, \mathbf{z}))
P_n^{ (\alpha,\beta)}  (\cos(\rho(\mathbf{x}_2,\mathbf{z})))\,\mathrm{d}\mathbf{z}\\
& = & P_n^{(\alpha,\beta)} (\cos(\rho(\mathbf{x}_1, \mathbf{x}_2) ),
\end{eqnarray*}
by Lemma~\ref{lem:1}. Equation (\ref{eq:6})  easily follows from the same lemma.
\end{proof}

\begin{proof}[Proof of Theorem~\ref{th:1}]
The series at the right-hand side of (\ref{spatiostoc1}) converges in mean square for every $\mathbf{x} \in \mathbb{M}^d$ since
      \begin{eqnarray*}
      &  &  \mathsf{E}\left[  \left(  \sum_{i=n_1}^{n_1+n_2}   \mathsf{B}_i^{\frac{1}{2}} \mathbf{V}_i P_i^{ (\alpha, \beta) } ( \cos \rho (\mathbf{x},  \mathbf{U} )) \right)
                       \left(  \sum_{j=n_1}^{n_1+n_2}   \mathsf{B}_j^{\frac{1}{2}} \mathbf{V}_j  P_j^{ (\alpha, \beta) } ( \cos \rho (\mathbf{x},  \mathbf{U} )) \right)^{\top}\right] \\
      & = &         \sum_{i=n_1}^{n_1+n_2}     \sum_{j=n_1}^{n_1+n_2}
                        \mathsf{B}_i^{\frac{1}{2}} \mathsf{B}_j^{\frac{1}{2}}  \mathsf{E}[ ( \mathbf{V}_i \mathbf{V}^{\top}_j)]
                        \mathsf{E} \left[\left(  P_i^{ (\alpha, \beta) } ( \cos \rho (\mathbf{x},  \mathbf{U} ))   P_j^{ (\alpha, \beta) } ( \cos \rho (\mathbf{x},  \mathbf{U} )) \right)\right] \\
      & = &    \sum_{i=n_1}^{n_1+n_2}   \mathsf{B}_i \sigma_i^2    \mathsf{E} \left[\left(  P_i^{ (\alpha, \beta) } ( \cos \rho (\mathbf{x},  \mathbf{U} ))   P_i^{ (\alpha, \beta) } ( \rho (\mathbf{x},  \mathbf{U} )) \right)\right] \\
      & = &    \sum_{i=n_1}^{n_1+n_2}   \mathsf{B}_i    P_i^{ (\alpha, \beta) } ( 1)  \\
      & \to & \mathbf{0},  ~~~~ \mbox{as} ~ n_1, n_2 \to \infty,
      \end{eqnarray*}
      where the second equality follows from the independent assumption between $\{\, \mathbf{V}_n\colon n \in \mathbb{N}_0\, \}$ and $\mathbf{U}$, and the fourth from Lemma~\ref{lem:2}. Thus (\ref{spatiostoc1}) is an $m$-variate second-order random field.
      Its mean function is clearly identical to $\mathbf{0}$, and it covariance function is
          \begin{eqnarray*}
          &  &  \cov \left( \sum_{i=0}^\infty  \mathsf{B}_i^{\frac{1}{2}} \mathbf{V}_i P_i^{ (\alpha, \beta) } ( \cos \rho (\mathbf{x}_1,  \mathbf{U} )),
                                  ~  \sum_{j=0}^\infty  \mathsf{B}_j^{\frac{1}{2}} \mathbf{V}_j P_j^{ (\alpha, \beta) } ( \cos \rho (\mathbf{x}_2,  \mathbf{U} )) \right) \\
          & = &     \sum_{i=0}^\infty     \sum_{j=0}^\infty
                       \mathsf{B}_i^{\frac{1}{2}} \mathsf{B}_j^{\frac{1}{2}}   \mathsf{E}\left[( \mathbf{V}_i \mathbf{V}^{\top}_j)]
                        \mathsf{E}[ \left(  P_i^{ (\alpha, \beta)}  ( \cos \rho (\mathbf{x},  \mathbf{U} ))   P_j^{ (\alpha, \beta) } ( \cos \rho (\mathbf{x},  \mathbf{U} )) \right)\right] \\
          & = &     \sum_{i=0}^\infty
                       \mathsf{B}_i \sigma_i^2
                        \mathsf{E} \left[\left(  P_i^{ (\alpha, \beta) } ( \cos \rho (\mathbf{x}_1,  \mathbf{U} ))   P_i^{ (\alpha, \beta) } ( \cos \rho (\mathbf{x}_2,  \mathbf{U} )) \right)\right] \\
          & = &       \sum_{i=0}^\infty
                       \mathsf{B}_i        P_i^{ (\alpha, \beta) } ( \cos \rho (\mathbf{x}_1,  \mathbf{x}_2)),  ~~~~ \mathbf{x}_1, \mathbf{x}_2 \in \mathbb{M}^d.
          \end{eqnarray*}
      Two distinct terms of (\ref{spatiostoc1}) are obviously uncorrelated each other.
\end{proof}

\begin{proof}[Proof of Theorem~\ref{th:2}]
It suffices to verify (\ref{spatio.cmf2}) to be a general form, since in Theorem~\ref{th:1} we already construct an $m$-variate isotropic random field on $\mathbb{M}^d$ whose covariance matrix function is (\ref{spatio.cmf2}).  To this end, suppose that $\{\, \mathbf{Z}(\mathbf{x})\colon \mathbf{x} \in \mathbb{M}^d\, \}$ is an $m$-variate isotropic and mean square continuous random
   field. Then, for an arbitrary $\mathbf{a} \in \mathbb{R}^m$, $\{\, \mathbf{a}^{\top} \mathbf{Z}(\mathbf{x})\colon \mathbf{x} \in \mathbb{M}^d\, \}$ is a scalar isotropic and mean square continuous random
   field, so that its covariance function has to be of the form (\ref{scalar.cov}),
       \begin{equation}
       \label{eq.Th2.proof1}
          \cov ( \mathbf{a}^{\top} \mathbf{Z}(\mathbf{x}_1), \mathbf{a}^{\top} \mathbf{Z}(\mathbf{x}_2) )
        = \sum_{n=0}^\infty b_n (\mathbf{a})
            P_n^{ (\alpha, \beta) } ( \cos \rho (\mathbf{x}_1,  \mathbf{x}_2)),  ~~~~ \mathbf{x}_1, \mathbf{x}_2 \in \mathbb{M}^d,
          \end{equation}
       where $\{\, b_n (\mathbf{a})\colon n \in \mathbb{N}_0\, \}$ is a sequence of nonnegative constants and $\sum\limits_{n=0}^\infty b_n (\mathbf{a})
            P_n^{ (\alpha, \beta) } ( 1 )$ converges. Similarly, for $\mathbf{b} \in \mathbb{R}^m$, we obtain
         \begin{eqnarray*}
       &  &  \frac{1}{4} \cov ( (\mathbf{a} +\mathbf{b})^{\top} \mathbf{Z}(\mathbf{x}_1), (\mathbf{a}+\mathbf{b})^{\top} \mathbf{Z}(\mathbf{x}_2) )
        = \sum_{n=0}^\infty b_n (\mathbf{a}+\mathbf{b})
            P_n^{ (\alpha, \beta) } ( \cos \rho (\mathbf{x}_1,  \mathbf{x}_2)),  \\
       &  &   \frac{1}{4} \cov ( (\mathbf{a} -\mathbf{b})^{\top} \mathbf{Z}(\mathbf{x}_1), (\mathbf{a}-\mathbf{b})^{\top} \mathbf{Z}(\mathbf{x}_2) )
        = \sum_{n=0}^\infty b_n (\mathbf{a}-\mathbf{b})
            P_n^{ (\alpha, \beta) } ( \cos \rho (\mathbf{x}_1,  \mathbf{x}_2)),  ~~ \mathbf{x}_1, \mathbf{x}_2 \in \mathbb{M}^d.
         \end{eqnarray*}
       Taking the difference between the last two equations yields
           \begin{eqnarray*}
           &  & \frac{1}{2} ( \mathbf{a}^{\top}  \cov ( \mathbf{Z}(\mathbf{x}_1),  \mathbf{Z}(\mathbf{x}_2) ) \mathbf{b}+
                                      \mathbf{b}^{\top}  \cov (  \mathbf{Z}(\mathbf{x}_1),  \mathbf{Z}(\mathbf{x}_2) ) \mathbf{a} ) \\
         & = & \frac{1}{2} ( \cov ( \mathbf{a}^{\top} \mathbf{Z}(\mathbf{x}_1), \mathbf{b}^{\top} \mathbf{Z}(\mathbf{x}_2) ) +\cov ( \mathbf{b}^{\top} \mathbf{Z}(\mathbf{x}_1), \mathbf{a}^{\top} \mathbf{Z}(\mathbf{x}_2) )) \\
         & = & \sum\limits_{n=0}^\infty (b_n (\mathbf{a}+\mathbf{b}) -b_n (\mathbf{a}-\mathbf{b}))
            P_n^{ (\alpha, \beta) } ( \cos \rho (\mathbf{x}_1,  \mathbf{x}_2)),  ~~ \mathbf{x}_1, \mathbf{x}_2 \in \mathbb{M}^d,
            \end{eqnarray*}
       or
         \begin{equation}
       \label{eq.Th2.proof2}
\mathbf{a}^{\top}  \cov (  \mathbf{Z}(\mathbf{x}_1),  \mathbf{Z}(\mathbf{x}_2) ) \mathbf{b}
        = \sum\limits_{n=0}^\infty (b_n (\mathbf{a}+\mathbf{b}) -b_n (\mathbf{a}-\mathbf{b}))
            P_n^{ (\alpha, \beta) } ( \cos \rho (\mathbf{x}_1,  \mathbf{x}_2)),  ~~ \mathbf{x}_1, \mathbf{x}_2 \in \mathbb{M}^d,
       \end{equation}
       noticing that $\cov (  \mathbf{Z}(\mathbf{x}_1),  \mathbf{Z}(\mathbf{x}_2) ) $ is a symmetric matrix.
      The form (\ref{spatio.cmf2}) of $\cov (  \mathbf{Z}(\mathbf{x}_1),  \mathbf{Z}(\mathbf{x}_2) ) $  is now confirmed by letting the $i$th entry of $\mathbf{a}$ and the $j$th entry of $\mathbf{b}$ be 1 and the rest vanish
      in (\ref{eq.Th2.proof2}).
      It remains to verify the nonnegative definiteness of each $\mathsf{B}_n$ in (\ref{spatio.cmf2}). To do so, we multiply its both sides by $\mathbf{a}^{\top}$ from the left and $\mathbf{a}$ from the right, and obtain
      $$  \mathbf{a}^{\top} \mathsf{C} ( \mathbf{x}_1, \mathbf{x}_2 )  \mathbf{a} = \sum_{n=0}^\infty  \mathbf{a}^{\top}  \mathsf{B}_n  \mathbf{a} P_n^{ (\alpha, \beta) } \left( \cos  \rho (\mathbf{x}_1, \mathbf{x}_2)  \right), ~~   \mathbf{x}_1, \mathbf{x}_2 \in \mathbb{M}^d, $$
      comparing which with (\ref{eq.Th2.proof1}) results in that $\mathbf{a}^{\top}  \mathsf{B}_n  \mathbf{a} \ge 0$ or  the nonnegative definiteness of  $\mathsf{B}_n$,
      $n \in \mathbb{N}_0$, and the convergence of $\sum\limits_{n=0}^\infty  \mathbf{a}^{\top}  \mathsf{B}_n  \mathbf{a} P_n^{ (\alpha, \beta)} (1)$ or
      that of each entry of  the matrix $\sum\limits_{n=0}^\infty  \mathsf{B}_n   P_n^{ (\alpha, \beta) } (1)$.
\end{proof}

\begin{proof}[Proof of Theorem~\ref{th:3}]
For a fixed $t \in \mathbb{T}$, consider a random field $\left\{\,  \mathbf{Z} (\mathbf{x}; 0) + \mathbf{Z} (\mathbf{x}; t)\colon  \mathbf{x} \in \mathbb{M}^d \,\right\}$.
It is isotropic and mean square continuous on $\mathbb{M}^d$, with  covariance matrix function
\[
\begin{aligned}
\cov \left(     \mathbf{Z} (\mathbf{x}_1; 0) + \mathbf{Z} (\mathbf{x}_1; t), ~ \mathbf{Z} (\mathbf{x}_2; 0) + \mathbf{Z} (\mathbf{x}_2; t)       \right)
&=\mathsf{C} (\rho (\mathbf{x}_1, \mathbf{x}_2); 0) +  \mathsf{C} (\rho (\mathbf{x}_1, \mathbf{x}_2); t) +  \mathsf{C} (\rho (\mathbf{x}_1, \mathbf{x}_2); -t)\\
&=\sum_{n=0}^\infty \mathsf{B}_{n+} (t)   P_n^{(\alpha, \beta) } (\cos \rho (\mathbf{x}_1, \mathbf{x}_2)),   ~~ \mathbf{x}_1, \mathbf{x}_2 \in \mathbb{M}^d,
\end{aligned}
\]
where the last equality follows from Theorem~\ref{th:2},  $\{\, \mathsf{B}_{n+} (t)\colon n \in \mathbb{N}_0\, \}$ is a sequence of nonnegative-definite matrices,
      and  $   \sum\limits_{n=0}^\infty \mathsf{B}_{n+} (t)   P_n^{(\alpha, \beta)} (1)$ converges. Similarly, we have
\[
\begin{aligned}
\cov \left(     \mathbf{Z} (\mathbf{x}_1; 0) - \mathbf{Z} (\mathbf{x}_1; t), ~ \mathbf{Z} (\mathbf{x}_2; 0) - \mathbf{Z} (\mathbf{x}_2; t)       \right)
&=2  \mathsf{C} (\rho (\mathbf{x}_1, \mathbf{x}_2); 0) -  \mathsf{C} (\rho (\mathbf{x}_1, \mathbf{x}_2); t) -  \mathsf{C} (\rho (\mathbf{x}_1, \mathbf{x}_2); -t)\\
&=\sum_{n=0}^\infty \mathsf{B}_{n-} (t)   P_n^{ (\alpha, \beta) } (\cos \rho (\mathbf{x}_1, \mathbf{x}_2)),
\end{aligned}
\]
     and, thus,
\[
\begin{aligned}
\frac{\mathsf{C} (\rho (\mathbf{x}_1, \mathbf{x}_2); t) +  \mathsf{C} (\rho (\mathbf{x}_1, \mathbf{x}_2); -t)}{2}&=\frac{1}{4} [ 2  \mathsf{C} (\rho (\mathbf{x}_1, \mathbf{x}_2); 0) +  \mathsf{C} (\rho (\mathbf{x}_1, \mathbf{x}_2); t) +  \mathsf{C} (\rho (\mathbf{x}_1, \mathbf{x}_2); -t)]\\
&\quad - \frac{1}{4} [  2  \mathsf{C} (\rho (\mathbf{x}_1, \mathbf{x}_2); 0) -  \mathsf{C} (\rho (\mathbf{x}_1, \mathbf{x}_2); t) -  \mathsf{C} (\rho (\mathbf{x}_1, \mathbf{x}_2); -t) ] \\
&=\sum_{n=0}^\infty \mathsf{B}_n (t)    P_n^{(\alpha, \beta)} (\cos \rho (\mathbf{x}_1, \mathbf{x}_2)),  ~~ \mathbf{x}_1, \mathbf{x}_2 \in \mathbb{M}^d,
\end{aligned}
\]
        which confirms the format (\ref{cov.matrix.fun.1}) for $\frac{\mathsf{C} (\rho (\mathbf{x}_1, \mathbf{x}_2); t) +  \mathsf{C} (\rho (\mathbf{x}_1, \mathbf{x}_2); -t)}{2}$, with
$ B_n (t) =\frac{ \mathsf{B}_{n+} (t) - \mathsf{B}_{n-} (t)}{4}$, $n \in \mathbb{N}_0$.
    Obviously, $\mathsf{B}_n(t)$ is symmetric,
        and   $\sum\limits_{n=0}^\infty  \mathsf{B}_n (t)  P_n^{ (\alpha, \beta) } (1)$ converges.
        Moreover,   (\ref{cov.matrix.fun.1}) is  the covariance matrix function of  an $m$-variate  isotropic random field $\left\{\,  \frac{ \mathbf{Z} (\mathbf{x}; t)+\tilde{\mathbf{Z}} (\mathbf{x}; -t)}{\sqrt{2}}\colon \mathbf{x} \in \mathbb{M}^d, t \in \mathbb{T}\, \right\}$, where    $\{ \,\tilde{\mathbf{Z}} (\mathbf{x}; t)\colon \mathbf{x} \in \mathbb{M}^d, t \in \mathbb{T} \,\}$  is    an independent copy of  $\{\, \mathbf{Z} (\mathbf{x}; t)\colon \mathbf{x}  \in \mathbb{M}^d, t \in \mathbb{T}\, \}$. In fact,
\[
\begin{aligned}
\cov \left( \frac{ \mathbf{Z} (\mathbf{x}_1; t_1)+\tilde{\mathbf{Z}} (\mathbf{x}_1; -t_1)}{\sqrt{2}}, ~  \frac{ \mathbf{Z} (\mathbf{x}_2; t_2)+\tilde{\mathbf{Z}} (\mathbf{x}_2; -t_2)}{\sqrt{2}}  \right)&=\frac{ \mathsf{C} (\rho (\mathbf{x}_1, \mathbf{x}_2); t_1-t_2)+\mathsf{C} ( \rho (\mathbf{x}_1, \mathbf{x}_2); t_2-t_1)}{2}   \\
&=\sum_{k=0}^\infty \mathsf{B}_{k} (t_1-t_2)   P_k^{ (\alpha, \beta) } (\cos \rho (\mathbf{x}_1, \mathbf{x}_2))
\end{aligned}
\]
with $\mathbf{x}_1$, $\mathbf{x}_2 \in \mathbb{M}^d$, $t_1$, $t_2 \in \mathbb{T}$.

    For each fixed $n \in \mathbb{N}_0$,   in order to verify that $ \mathsf{B}_n (t) $ is a  stationary covariance matrix function on $\mathbb{T}$,
        we consider an $m$-variate  stochastic process
              $$ \mathbf{W}_n (t) = \int_{\mathbb{M}^d}  \frac{ \mathbf{Z} (\mathbf{x}; t)+\tilde{\mathbf{Z}} (\mathbf{x}; -t)}{\sqrt{2}}  P_n^{ (\alpha, \beta)  } (  \cos \rho (\mathbf{x}, \mathbf{U}) ) d \mathbf{x},  ~~~~~~~~~~~~~~ t \in \mathbb{T}, $$
        where $\{\, \tilde{\mathbf{Z}} (\mathbf{x}; t)\colon \mathbf{x} \in \mathbb{M}^d, t \in \mathbb{T}\, \}$ is an independent copy of  $\{\, \mathbf{Z} (\mathbf{x}; t)\colon \mathbf{x}  \in \mathbb{M}^d, t \in \mathbb{T}\, \}$,  $\mathbf{U}$ is a random vector uniformly distributed on $\mathbb{M}^d$,  and $\mathbf{U}$,  $\{\, \mathbf{Z} (\mathbf{x}; t)\colon \mathbf{x}  \in \mathbb{M}^d, t \in \mathbb{T}\, \}$ and  $\{\, \tilde{\mathbf{Z}} (\mathbf{x}; t)\colon \mathbf{x} \in \mathbb{S}^d, t \in \mathbb{T}\, \}$ are independent.
       By Lemma~\ref{lem:1},
         the  mean function of $\{\, \mathbf{W}_n (t)\colon  t \in \mathbb{T}\, \}$  is
          \begin{eqnarray*}
           \mathsf{E} [  \mathbf{W}_n (t)]
             & =  &  \left\{
                          \begin{array}{ll}
                          \sqrt{2} P^{(\alpha,\beta)}_0(1)\omega_d  \mathsf{E} [\mathbf{Z} (\mathbf{x}; t)],   ~   &    ~  n= 0,   \\
                          0,     ~   &     ~  n \in \mathbb{N},
                          \end{array}    \right.
                  \end{eqnarray*}
            and its   covariance matrix function is, by Lemmas~\ref{lem:1} and \ref{lem:2}
\[
\begin{aligned}
&\cov ( \mathbf{W}_n(t_1), ~  \mathbf{W}_n (t_2) ) = \frac{1}{\omega_d} \cov   \left(       \int_{\mathbb{M}^d}  \frac{ \mathbf{Z} (\mathbf{x}; t_1)+\tilde{\mathbf{Z}} (\mathbf{x}; -t_1)}{\sqrt{2}}  P_n^{ (\alpha, \beta)  } ( \cos  \rho( \mathbf{x}, \mathbf{U}) )  d \mathbf{x},\right.\\
&\quad\left.\int_{\mathbb{M}^d}  \frac{ \mathbf{Z} (\mathbf{y}; t_2)+\tilde{\mathbf{Z}} (\mathbf{y}; -t_2)}{\sqrt{2}}  P_n^{ (\alpha, \beta) } ( \cos  \rho( \mathbf{y}, \mathbf{U})) d \mathbf{y}                           \right) \\
&=\frac{1}{\omega_d} \int_{\mathbb{M}^d}    \cov   \left(       \int_{\mathbb{M}^d}  \frac{ \mathbf{Z} (\mathbf{x}; t_1)+\tilde{\mathbf{Z}} (\mathbf{x}; -t_1)}{\sqrt{2}}  P_n^{ (\alpha, \beta)  } ( \cos  \rho( \mathbf{x}, \mathbf{U}) )d \mathbf{x}, ~ \right. \\
&\quad\left.~~~~~~~~~~       \int_{\mathbb{M}^d}  \frac{ \mathbf{Z} (\mathbf{y}; t_2)+\tilde{\mathbf{Z}} (\mathbf{y}; -t_2)}{\sqrt{2}}  P_n^{ (\alpha, \beta) } (\cos  \rho( \mathbf{y}, \mathbf{u}) ) d \mathbf{y}                           \right)   d \mathbf{u}  \\
& =\frac{1}{\omega_d}   \int_{\mathbb{M}^d}      \int_{\mathbb{M}^d}        \int_{\mathbb{M}^d}   \cov \left( \frac{ \mathbf{Z} (\mathbf{x}; t_1)+\tilde{\mathbf{Z}} (\mathbf{x}; -t_1)}{\sqrt{2}}, ~  \frac{ \mathbf{Z} (\mathbf{y}; t_2)+\tilde{\mathbf{Z}} (\mathbf{y}; -t_2)}{\sqrt{2}}  \right)  \\
&\quad\times P_n^{ (\alpha, \beta) } (\cos  \rho( \mathbf{x}, \mathbf{u}) )    P_n^{ (\alpha, \beta) } (\cos  \rho( \mathbf{y}, \mathbf{u}) )  d \mathbf{x} d \mathbf{y}                             d \mathbf{u}  \\
&=\int_{\mathbb{M}^d}         \int_{\mathbb{M}^d}     \int_{\mathbb{M}^d}    \frac{ \mathsf{C} (\rho (\mathbf{x}, \mathbf{y}); t_1-t_2)+\mathsf{C} ( \rho (\mathbf{x}, \mathbf{y}); t_2-t_1)}{2\omega_d}
                       P_n^{ (\alpha, \beta) } ( \cos  \rho( \mathbf{x}, \mathbf{u}) )    P_n^{ (\alpha, \beta) } ( \cos  \rho( \mathbf{y}, \mathbf{u}) )  d \mathbf{x} d \mathbf{y}                             d \mathbf{u}  \\
\end{aligned}
\]
\[
\begin{aligned}
&=\frac{1}{\omega_d}    \int_{\mathbb{M}^d}        \int_{\mathbb{M}^d}      \int_{\mathbb{M}^d}      \sum_{k=0}^\infty \mathsf{B}_{k} (t_1-t_2)   P_k^{ (\alpha, \beta) } ( \cos  \rho( \mathbf{x}, \mathbf{y}))
                        P_n^{ (\alpha, \beta) } (\cos  \rho( \mathbf{x}, \mathbf{u}))    P_n^{ (\alpha, \beta) } (\cos  \rho( \mathbf{y}, \mathbf{u}))  d \mathbf{x} d \mathbf{y}                             d \mathbf{u}  \\
&=\frac{1}{\omega_d}        \sum_{k=0}^\infty \mathsf{B}_{k} (t_1-t_2)      \int_{\mathbb{M}^d}     \int_{\mathbb{M}^d}        \int_{\mathbb{M}^d}
                         P_k^{ (\alpha, \beta) } (\cos  \rho( \mathbf{x}, \mathbf{y}))
                        P_n^{ (\alpha, \beta) } ( \cos  \rho( \mathbf{x}, \mathbf{u}) )   d \mathbf{x}   P_n^{ (\alpha, \beta) } (\cos  \rho( \mathbf{y}, \mathbf{u}) )   d \mathbf{y}               d \mathbf{u}  \\
&= \frac{1}{\omega_d}       \mathsf{B}_{n} (t_1-t_2)      \int_{\mathbb{M}^d}   \frac{1}{a_n^2}  \int_{\mathbb{M}^d}
                         P_n^{ (\alpha, \beta) } (\cos  \rho( \mathbf{y}, \mathbf{u}))   P_n^{ (\alpha, \beta) } (\cos  \rho( \mathbf{y}, \mathbf{u}))   d \mathbf{y}                    d \mathbf{u}  \\
&= \frac{1}{\omega_d}     \mathsf{B}_{n} (t_1-t_2)      \int_{\mathbb{M}^d}   \left(  \frac{\omega_d}{a_n^2} \right)^2
                         P_n^{(\alpha, \beta) } (1)                d \mathbf{u}  \\
&=\mathsf{B}_{n} (t_1-t_2)       \left(  \frac{\omega_d}{a_n^2} \right)^2
                         P_n^{ (\alpha, \beta) } (1),    ~~~~~~ t_1, t_2 \in \mathbb{T},
\end{aligned}
\]
             which implies that  $ \mathsf{B}_n (t)$ is a  stationary covariance matrix function on $\mathbb{T}$.
\end{proof}

\begin{proof}[Proof of Theorem~\ref{th:4}]
The convergent assumption of  $\sum\limits_{n=0}^\infty  \mathsf{B}_n (0) P_n^{ (\alpha, \beta) } (1)$ ensures the uniform and absolute  convergence of the series at the right-hand side of (\ref{cov.matrix.fun.2}). If $\{\, \mathsf{B}_n (t)\colon  n \in \mathbb{N}_0\, \}$ is a sequence of stationary covariance matrix function on $\mathbb{T}$, then each term of  the series at the right-hand side of (\ref{cov.matrix.fun.2}) is the product of a
stationary covariance matrix function $\mathsf{B}_n (t)$ on $\mathbb{T}$ and an isotropic covariance   function $P_n^{ (\alpha, \beta) } (\cos \rho (\mathbf{x}_1, \mathbf{x}_2)$
on $\mathbb{M}^d$, and thus      (\ref{cov.matrix.fun.2}) can be treated  \cite{MR2774237} as the covariance matrix function of an $m$-variate random field on $\mathbb{M}^d \times \mathbb{T}$.

            On the other hand,         assume that (\ref{cov.matrix.fun.2}) is the covariance matrix function of an $m$-variate random field $\{\, \mathbf{Z} (\mathbf{x}; t)\colon \mathbf{x} \in \mathbb{M}^d, t \in \mathbb{T}\, \}$.
                The convergence of   $\sum\limits_{n=0}^\infty  \mathsf{B}_n (0)  P_n^{ (\alpha, \beta) } (1)$ results from
               the existence of $\mathsf{C} (0; 0) = \var [Z (\mathbf{x}; t)]$.
             In order to show that $\mathsf{B}_n (t)$ is a stationary covariance matrix function on $\mathbb{T}$  for each fixed $n \in \mathbb{N}_0$,  consider an $m$-variate  stochastic process
                  $$ \mathbf{W}_n (t) = \int_{\mathbb{M}^d} \mathbf{Z} (\mathbf{x}; t)  P_n^{   (\alpha, \beta) } ( \cos \rho (\mathbf{x}, \mathbf{U}))  d \mathbf{x},  ~~~~~~~~~ t \in \mathbb{T}, $$
              where $\mathbf{U}$ is a random vector uniformly distributed on $\mathbb{M}^d$ and is  independent with     $\{\, \mathbf{Z} (\mathbf{x}; t)\colon \mathbf{x} \in \mathbb{M}^d, t \in \mathbb{T} \,\}$.
              Similar to  that in the proof of Theorem~\ref{th:3},  applying Lemmas~\ref{lem:1} and \ref{lem:2} we  obtain that the covariance matrix function of $\{\, \mathbf{W}_n (t)\colon t \in \mathbb{T}\, \}$ is positively propositional to
              $ \mathsf{B}_n (t)$; more precisely,
                   \begin{eqnarray*}
                    \cov ( \mathbf{W}_n (t_1),   \mathbf{W}_n (t_2) )
                   & =  &    \mathsf{B}_{n} (t_1-t_2)       \left(  \frac{\omega_d}{a_n^2} \right)^2
                         P_n^{  (\alpha, \beta)  } (1),    ~~~~~~ t_1, t_2 \in \mathbb{T},
                   \end{eqnarray*}
            which implies that  $ \mathsf{B}_n (t)$ is a  stationary covariance matrix function on $\mathbb{T}$.
\end{proof}

\begin{proof}[Proof of Theorem~\ref{th:5}]
The convergent  assumption  of $\sum\limits_{n=0}^\infty \mathsf{B}_n (0)  P_n^{ (\alpha, \beta)  } (1)$
       ensures  the mean square convergence of the series at the right hand of (\ref{stoc1}), since
           \begin{eqnarray*}
      &  &      \mathsf{E}\left[ \left( \sum_{i=n_1}^{n_1+n_2}  \mathbf{V}_i (t)  P_i^{ (\alpha, \beta) } ( \cos  \rho( \mathbf{x}, \mathbf{U}) ) \right)
                  \left( \sum_{j=n_1}^{n_1+n_2}  \mathbf{V}_j (t) P_j^{ (\alpha, \beta)} ( \cos  \rho( \mathbf{x}, \mathbf{U})) \right)^{\top}\right] \\
      & =  &      \mathsf{E} \left[  \sum_{i=n_1}^{n_1+n_2}  \sum_{j=n_1}^{n_1+n_2}
           \mathbf{V}_i (t)   \mathbf{V}^{\top}_j  (t)   P_i^{ (\alpha, \beta) } (\cos  \rho( \mathbf{x}, \mathbf{U}))  P_j^{ (\alpha, \beta) } (\cos  \rho( \mathbf{x}, \mathbf{U})) \right] \\
      & =  &       \sum_{i=n_1}^{n_1+n_2}  \sum_{j=n_1}^{n_1+n_2}
           \mathsf{E} [\mathbf{V}_i  (t)  \mathbf{V}^{\top}_j (t) ]   \mathsf{E} \left[ P_i^{ (\alpha, \beta) } (\cos  \rho( \mathbf{x}, \mathbf{U}))  P_j^{ (\alpha, \beta) } (\cos  \rho( \mathbf{x}, \mathbf{U}))   \right] \\
      & =  &  \omega_d  \sum_{i=n_1}^{n_1+n_2}
           \mathsf{B}_i (0)    P_i^{ (\alpha, \beta) } (1)    \\
      &   \to  &  0,     ~~~~~~~  \mbox{as} ~  n_1, n_2 \to \infty,
            \end{eqnarray*}
              where the second equality follows from   the independent assumption between $\mathbf{U}$ and $\{\, \mathbf{V}_n (t)\colon  n \in \mathbb{N}_0\, \}$,  and  the third  one from Lemma~\ref{lem:2}.
 Applying Lemma~\ref{lem:2} we obtain the mean and covariance matrix functions of
  $\{\, \mathbf{Z} (\mathbf{x}; t)\colon \mathbf{x} \in \mathbb{M}^d, t \in \mathbb{T}\, \}$,  under the independent assumption among $\mathbf{U}$ and $\{\, \mathbf{V}_n (t)\colon n \in \mathbb{N}_0\, \}$,
     $$ \mathsf{E} [\mathbf{Z} (\mathbf{x}; t )] = \sum_{n=0}^\infty     \mathsf{E} [\mathbf{V}_n  (t)] \mathsf{E} \left[P_n^{ (\alpha, \beta) } (\cos  \rho( \mathbf{x}, \mathbf{U}))\right]  = \mathbf{0},
        ~~~ \mathbf{x} \in \mathbb{M}^d,  t \in \mathbb{T}, $$
   and
\[
\begin{aligned}
\cov ( \mathbf{Z} (\mathbf{x}_1; t_1), \mathbf{Z} ( \mathbf{x}_2; t_2) ) &=
\cov \left(   \sum_{i=0}^\infty   \mathbf{V}_i (t_1)   P_i^{ (\alpha, \beta) } (\cos  \rho( \mathbf{x}_1, \mathbf{U})),
                                     ~ \sum_{j=0}^\infty   \mathbf{V}_j (t_2)  P_j^{ (\alpha, \beta) } (\cos  \rho( \mathbf{x}_2, \mathbf{U}))  \right)  \\
&=\sum_{i=0}^\infty  \sum_{j=0}^\infty    \mathsf{E} [ \mathbf{V}_i (t_1) \mathbf{V}^{\top}_j (t_2) ]
                        \mathsf{E} \left[
                                           P_i^{ (\alpha, \beta) } ( \cos  \rho( \mathbf{x}_1, \mathbf{U})) P_j^{ (\alpha, \beta) } (\cos  \rho( \mathbf{x}_2, \mathbf{U})) \right]   \\
&=\sum_{n=0}^\infty     \mathsf{B}_n  (t_1-t_2)\frac{1}{a^2_n}
                                           P_n^{ (\alpha, \beta) } ( \cos \rho (\mathbf{x}_1,  \mathbf{x}_2) ),  ~~~~~~~~~~ \mathbf{x}_1, \mathbf{x}_2 \in \mathbb{M}^d, ~ t_1, t_2 \in \mathbb{T}.
\end{aligned}
\]
                    The latter is obviously isotropic and continuous on $\mathbb{M}^d$ and stationary on $\mathbb{T}$. 
\end{proof}
}

\bibliographystyle{plain}      
\bibliography{\jobname}   

\end{document}